\def\mystyle{}

\if\mystyle l
\documentclass[landscape]{amsart}
\else
\documentclass{amsart}
\fi

\ifx\MyBeamer\undefined
\usepackage[breaklinks,colorlinks]{hyperref}
\usepackage[a4paper,margin=1in]{geometry}
\else
\if\MyBeamer t
\usepackage[breaklinks,colorlinks]{hyperref}
\usepackage[a4paper,margin=1in]{geometry}
\else
\def\emph{\alert}
\fi
\fi

\usepackage{amsmath}
\usepackage{amssymb}
\usepackage{mathrsfs}
\usepackage{amsthm}
\usepackage{aliascnt}
\usepackage[utf8]{inputenc}
\usepackage[T1]{fontenc}
\ifx\FrenchText\undefined
\usepackage[british]{babel}
\else
\usepackage[british,french]{babel}
\AtBeginDocument{%
\catcode`\:=12%
\catcode`\!=12%
}
\fi

\makeatletter

\newcounter{quotethmcnt}

\ifx\MyBeamer\undefined

\def\equationautorefname~#1\null{(#1)}
\def\itemautorefname~#1\null{#1}
\fi

\ifx\MyBeamer\undefined

\ifx\thmnum\undefined
\newtheorem{thmnum}{}[section]
\fi

\newcommand{\mynewthm}[3][]{%
  \newaliascnt{#2}{thmnum}%
  \newtheorem{#2}[#2]{#3}%
  \aliascntresetthe{#2}%
  \newtheorem*{#2*}{#3}%
  \expandafter\newcommand\csname #2autorefname\endcsname{#3}%
  \expandafter\renewcommand\csname the#2\endcsname{\thethmnum}%
}

\newtheorem*{clm}{Claim}
\newenvironment{clmprf}{%
  \begin{proof}[Proof of claim]%
  }{\end{proof}}

\else

\let\xxx=\frametitle
\def\frametitle#1{%
  \xxx{%
    \setbeamercolor*{math text}{use={titlelike,my math text},fg=titlelike.fg!80!my math text.fg}%
    #1}%
  \setbeamercolor{math text}{use=my math text,fg=my math text.fg}%
}

\newcommand{\beamerenv}[3]{%
\newenvironment<>{#1}%
{%
  \setbeamercolor{temp}{fg=structure.fg}%
  \setbeamercolor{structure}{fg=#2}%
  \setbeamercolor{block body}{use=structure,bg=structure.fg!5!white}%
  \begin{#3}%
}%
{\end{#3}\setbeamercolor{structure}{fg=temp.fg}}}

\newcommand{\mynewthm}[3][green!50!black]{%
  \newtheorem*{#2x}{#3}%
  \beamerenv{#2}{#1}{#2x}%
}

\beamerenv{other}{violet}{block}

\fi

\ifx\FrenchText\undefined
\newcommand{\myiffrench}[2]{#2}
\else
\newcommand{\myiffrench}[2]{\iflanguage{french}{#1}{#2}}
\fi

\theoremstyle{plain}
\mynewthm[red]{thm}{\myiffrench{Théorème}{Theorem}}
\mynewthm{prp}{Proposition}
\mynewthm[purple]{lem}{\myiffrench{Lemme}{Lemma}}
\mynewthm[orange]{fct}{\myiffrench{Fait}{Fact}}
\mynewthm[purple!50!white]{cor}{\myiffrench{Corollaire}{Corollary}}

\theoremstyle{definition}
\mynewthm[green!80!black]{dfn}{\myiffrench{Définition}{Definition}}
\mynewthm{hyp}{\myiffrench{Hypothèse}{Hypothesis}}
\mynewthm{conv}{Convention}
\mynewthm{conj}{Conjecture}
\mynewthm{ntn}{Notation}
\mynewthm{cst}{Construction}

\theoremstyle{remark}
\mynewthm[yellow!90!black]{rmk}{\myiffrench{Remarque}{Remark}}
\mynewthm{qst}{Question}
\mynewthm{exc}{\myiffrench{Exercice}{Exercise}}
\mynewthm{exm}{\myiffrench{Exemple}{Example}}

\ifx\MyBeamer\undefined

\newcommand{\myenumlabel}[1]{\textnormal{(\roman{#1})}}

\fi

\newcounter{cycprfcnt}
\newcounter{cycprffirst}
\newcommand{\cycprfpreamble}[1]%
{%
  \setcounter{cycprfcnt}{1}
  \setcounter{cycprffirst}{#1}
  \setlength{\itemindent}{0.5\leftmargin}%
  \setlength{\leftmargin}{0pt}%
  \newcommand{\cpcurr}{\myenumlabel{cycprfcnt}}%
  \newcommand{\cpnext}{\addtocounter{cycprfcnt}{1}\cpcurr}%
  \newcommand{\impnext}{\cpcurr{} $\Longrightarrow$ \cpnext.}%
  \def\makelabel##1{\ifnum\value{cycprffirst}=0\hspace{-0.7\itemindent}\setcounter{cycprffirst}{1}\fi##1}%
}%

{\begin{list}{\impnext}{\cycprfpreamble{#1}}}%
{\qedhere\end{list}}%

\newenvironment{cycprf*}[1][0]%
{\begin{list}{\impnext}{\cycprfpreamble{#1}}}%
{\end{list}}%

\def\indsym#1#2{%
  \setbox0=\hbox{$\m@th#1x$}%
  \kern\wd0%
  \hbox to 0pt{\hss$\m@th#1\mid$\hbox to 0pt{$\m@th#1^{#2}$\hss}\hss}%
  \lower.9\ht0\hbox to 0pt{\hss$\m@th#1\smile$\hss}%
  \kern\wd0}

\def\nindsym#1#2{%
  \setbox0=\hbox{$\m@th#1x$}%
  \kern\wd0%
  \hbox to 0pt{\hss$\m@th#1\not$\kern1.4\wd0\hss}
  \hbox to 0pt{\hss$\m@th#1\mid$\hbox to 0pt{$\m@th#1^{#2}$\hss}\hss}%
  \lower.9\ht0\hbox to 0pt{\hss$\m@th#1\smile$\hss}%
  \kern\wd0}

\def\dotminussym#1#2{%
  \setbox0=\hbox{$\m@th#1-$}%
  \kern.5\wd0%
  \hbox to 0pt{\hss\hbox{$\m@th#1-$}\hss}%
  \raise.6\ht0\hbox to 0pt{\hss$\m@th#1.$\hss}%
  \kern.5\wd0}

\renewcommand{\emptyset}{\varnothing}
\renewcommand{\setminus}{\smallsetminus}

\DeclareMathOperator{\Aut}{Aut}

\DeclareMathOperator{\Gal}{Gal}

\newcommand{\fF}{\mathfrak{F}}

\newcommand{\fH}{\mathfrak{H}}
\newcommand{\fI}{\mathfrak{I}}

\newcommand{\fL}{\mathfrak{L}}
\newcommand{\fM}{\mathfrak{M}}

\makeatother

\DeclareMathOperator{\codim}{codim}
\def\bbN{\mathbb N}
\def\Fix{\mbox{Fix}}
\def\Gal{\mbox{Gal}}
\def\aut{\mbox{Aut}}

\def\M{\mathfrak M}

\def\Ind#1#2{#1\setbox0=\hbox{$#1x$}\kern\wd0\hbox to 0pt{\hss$#1\mid$\hss}
  \lower.9\ht0\hbox to 0pt{\hss$#1\smile$\hss}\kern\wd0}

\def\Notind#1#2{#1\setbox0=\hbox{$#1x$}\kern\wd0\hbox to 0pt{\mathchardef
    \nn="3236\hss$#1\nn$\kern1.4\wd0\hss}\hbox to 0pt{\hss$#1\mid$\hss}\lower.9\ht0
  \hbox to 0pt{\hss$#1\smile$\hss}\kern\wd0}

\theoremstyle{plain}
\newtheorem{theorem}{Theorem}[section]
\newtheorem{proposition}[theorem]{Proposition}
\newtheorem{fact}[theorem]{Fact}
\newtheorem{lemma}[theorem]{Lemma}
\newtheorem{corollary}[theorem]{Corollary}
\newtheorem*{claim}{Claim}

\theoremstyle{definition}
\newtheorem{definition}[theorem]{Definition}
\newtheorem{remark}[theorem]{Remark}
\newtheorem{expl}[theorem]{Example}

\def\bsp{\begin{expl}}
  \def\ebsp{\end{expl}}
\def\beh{\begin{claim}}
  \def\ebeh{\end{claim}}
\def\defn{\begin{definition}}
  \def\edefn{\end{definition}}
\def\satz{\begin{theorem}}
  \def\esatz{\end{theorem}}
\def\tats{\begin{fact}}
  \def\etats{\end{fact}}
\def\kor{\begin{corollary}}
  \def\ekor{\end{corollary}}
\def\lmm{\begin{lemma}}
  \def\elmm{\end{lemma}}
\def\bem{\begin{remark}}
  \def\ebem{\end{remark}}
\def\bew{\begin{proof}}
  \def\ebew{\end{proof}}

\def\satzli{\begin{proposition}}
  \def\esatzli{\end{proposition}}

\begin{document}
\title{A metric version of Schlichting's Theorem}
\author{Ita\"\i\ Ben Yaacov and Frank O.\ Wagner}

\address{Universit\'e de Lyon; CNRS; Universit\'e Claude Bernard Lyon 1; Institut Camille Jordan UMR5208, 43 bd du 11 novembre 1918, 69622 Villeurbanne Cedex, France}

\urladdr{\url{http://math.univ-lyon1.fr/~begnac/}}
\urladdr{\url{http://math.univ-lyon1.fr/~wagner/}}
\email{benyaacov@math.univ-lyon1.fr}
\email{wagner@math.univ-lyon1.fr}

\date{\today}
\thanks{Partially supported by ANR-13-BS01-0006 ValCoMo}
\begin{abstract}If $\fF$ is a type-definable family of commensurable subsets, subgroups or sub-vector spaces in a metric structure, then there is an invariant subset, subgroup or sub-vector space commensurable with $\fF$. This in particular applies to type-definable or hyper-definable objects in a classical first-order structure.\end{abstract}
\maketitle

\section*{Introduction}

Schlichting's Theorem \cite{sch} states that if a subgroup $H$ of a group $G$ is uniformly commensurable with all its $G$-conjugates, then it is commensurable with a normal subgroup of $G$. This was generalized by Bergman and Lenstra \cite{bl}, who showed that if $H$ is uniformly commensurable with all its $K$-conjugates for some subgroup $K$ of $G$, then it is commensurable with a $K$-invariant subgroup. Peter Neumann deduced from this an analogous theorem for sets: A family of subsets of some set $\Omega$ invariant under
a subgroup $K$ of Sym$(\Omega)$ with bounded symmetric differences
yields a $K$-invariant subset whose symmetric difference with
members of the original family is bounded. This was studied further by Brailovsky, Pasechnik and Praeger \cite{bpp}, Neumann \cite{neu} and the second author \cite{wa98}, who proved a version for vector spaces, as well as more general objects.

Meanwhile, a similar theorem was shown for type-definable groups in simple theories, where the finite index condition of commensurability is re-interpreted as {\em bounded} index \cite[Theorem 4.5.13]{wa00} (building on results of Hrushovski for the S1-case). However, simplicity seemed a necessary condition, as the proof is based on the Independence Theorem.

Recently, the second author proved a hyperdefinable version of Schlichting's Theorem in \cite{wa16} without any hypotheses on the ambient theory. In this note we shall rephrase the result in the language of continuous logic and metric structures, and generalize it to families of sub-objects other than groups. As already in \cite{wa98} it will follow from a corresponding fixed point theorem for a certain kind of lattice.

It should be noted that in \cite{Udi} Hrushovski has generalized the so-called {\em Stabilizer Theorem}, which in simple theories is closely related to Schlichting's Theorem, to a much more general context, assuming the existence of an S1-ideal of ``small'' formul\ae. An intriguing question is thus to what extent our version of Schlichting's Theorem is related to Hrushovski's Stabilizer Theorem, and whether our approach might be used to generalize the Stabilizer Theorem even further.

Finally, in the last section we use Galois Theory to deduce a new version of Schlichting's Theorem for (finitely) commensurable fields.

\section{Close-knit Families}
Recall that Peter Neumann \cite{neu} calls a family $\fF$ of subsets of some ambient set {\em close-knit} if there is a finite bound on the cardinality of the difference $F\setminus F'$ for $F,F'\in\fF$. This was generalized by the second author \cite{wa98} to a family $\fF$ of points of a lattice with some integer-valued distance function $\delta$: it is {\em close-knit} if $\delta(F,F')$ is bounded for all $F,F'\in\fF$. In our set-up, the distance function on the lattice is no longer integer-valued; finiteness of the number of values is replaced by a suitable compactness condition.

\begin{dfn}
  \label{dfn:CloseKnit}
  Assume the following data are given:
  \begin{itemize}
  \item A $\kappa$-complete lower semi-lattice $\fL$ (i.e.\ every sub-family of size $<\kappa$ admits an infimum), for some regular cardinal $\kappa$.
  \item A family $\fF = \{f_a : a \in A\} \subseteq \fL$, where the enumeration may have repetitions.
  \item A compact Hausdorff topological space $\fI$ equipped with a closed partial order relation.
    We let $\lambda = w(\fI)^+ + \aleph_0$, where $w$ denotes the weight of $\fI$ (i.e., the least cardinal of an open basis), and require that $\kappa \geq \lambda$.
  \item A map $\delta\colon \fL \times A \rightarrow \fI$.
  \item A map $\fF_\kappa \times A \rightarrow \fL$ denoted $(s,a) \mapsto s^a$, where $\fF_\kappa \subseteq \fL$ denotes the family of meets of $<\kappa$ elements of $\fF$.
  \end{itemize}
  We say that $\fF$ (together with the additional data) is a \emph{close-knit family} if the following holds:
  \begin{enumerate}
  \item \label{item:CloseKnitDeltaContinuous}
    The map $\delta$ is monotonous and ``upper semi-continuous'' in its first argument in the sense that if $S \subseteq \fL$ is closed under finite meet and $|S| < \kappa$, then $\delta( \bigwedge S, a ) = \bigwedge \, \{\delta(s,a):s\in S\}$.
  \item \label{item:CloseKnitDeltaCompact}
    Let $\xi \in \fI$ and let $S \subseteq \fL$ be closed under finite meet such that $|S| < \kappa$.
    Assume that for every neighbourhood $\xi \in U \subseteq \fI$ and $s \in S$ there exists $\zeta \in U$ and $a \in A$ such that $\delta(s,a) \geq \zeta$.
    Then there exists $a \in A$ satisfying $\delta(s,a) \geq \xi$ for all $s \in S$.
    (In this case, $\delta\bigl( \bigwedge S, a \bigr) \geq \xi$ as well, by \autoref{item:CloseKnitDeltaContinuous}.)
  \item \label{item:CloseKnitIncrement}
    We have $s\le s^a$, and whenever $t \leq s$ are in $\fF_\kappa$ with $\delta(t,a)=\delta(s,a)$, then $t^a=s^a$.
  \item \label{item:CloseKnitLongDecreasingSequence}
    For any $s\in\fF_\lambda$, where $\fF_\lambda$ denotes the family of meets of $<\lambda$ elements of $\fF$, there is some cardinal $\mu_s < \kappa$ such that for all $a \in A$, any chain in $\fL$ between $s$ and $s^a$ has cardinality $<\mu_s$.
  \end{enumerate}
  If $\Gamma$ is a group of automorphisms of $\fL$, also acting on $A$, and all the data (namely, the maps $a \rightarrow f_a$, $\delta$ and $(s,a) \mapsto s^a$) are invariant under $\Gamma$, then we say that $\fF$ is a \emph{$\Gamma$-close-knit family}.
\end{dfn}

We should think of $\delta(s,a)$ as a measure of how much $s \nleq f_a$.
In particular, in most applications $s \leq f_a$ if and only if $\delta(s,a) = 0$.
The idea behind the map $(s,a) \mapsto s^a$ is to increase $s$ a little, toward being greater than $f_a$.
For example, in the next section, when $s$ and $f_a$ represent sets of small symmetric difference, we let $s^a = s \cup f_a$.
Similarly, when they represent vector spaces not much bigger than their intersection, we achieve the same effect with $s^a = s + f_a$.
However, unlike what these examples might suggest, we cannot actually require that $s^a \geq f_a$, since this fails in the setting of commensurable groups.

\begin{thm}
  \label{thm:CloseKnit}
  Let $\fL$ be a $\kappa$-complete lower semi-lattice, $\Gamma$ a group of automorphisms of $\fL$, and $\fF$ a $\Gamma$-close-knit family in $\fL$.
  Then $\Gamma$ has a fixed point in $\fL$.
\end{thm}
\begin{proof}
  For $s\in\fL$ define
  \begin{gather*}
    m(s) = \bigl\{ \xi \in \fI : \xi \leq \delta(s,a) \text{ for some } a \in A \bigr\}.
  \end{gather*}
\autoref{dfn:CloseKnit}\autoref{item:CloseKnitDeltaCompact} with $S=\{s\}$ implies that $m(s)$ is closed. As $\lambda= w(\fI)^+ + \aleph_0$, it follows that there is no strictly decreasing chain $(m(s_i):i<\lambda)$ with $s_i\in\fL$. But if $s,t\in\fL$ with $t\leq s$, then $\delta(t,a) \leq \delta(s,a)$ for all $a\in A$,  whence $m(t) \subseteq m(s)$. Hence $\{m(s):s\in\fF_\kappa\}$ has a unique minimal element $m$.
We call $s\in\fF_\kappa$ {\em strong} if $m(s)=m$.
 
If $S\subset\fL$ is closed under finite meets with $|S|<\kappa$ and $\xi \notin m(\bigwedge S)$, then by \autoref{dfn:CloseKnit}\autoref{item:CloseKnitDeltaCompact} there exists $s\in S$ such that $\xi \notin m(s)$. 
It follows that for any strong $s\in\fF_\kappa$ there is some strong $s'\in\fF_\lambda$ with $s\le s'$ (and in fact $s'$ is a subintersection of $s$).
  For strong $s$ we define
  \begin{align*}
    A(s) &= \bigl\{ a \in A : \delta(s,a) \text{ is maximal in } m \bigr\},\quad\mbox{and} \\
    n(s) &= \bigwedge \, \bigl\{ s^a : a \in A(s) \bigr\}.
  \end{align*}
  Since $m$ is closed, it contains maximal elements, so $A(s)$ is non-empty.
  Note that if $t \in \fF_\kappa$ with $t\le s$, then $m(t)\subseteq m(s)=m$, so $t$ is also strong by minimality of $m$, and $A(t) \subseteq A(s)$.
  If, in addition, we have $a \in A(t)$, then $\delta(t,a)=\delta(s,a)$ (since both are maximal in $m$, and they are comparable), whence $t^a=s^a$ by \autoref{dfn:CloseKnit}\autoref{item:CloseKnitIncrement}.
 In particular, if $\fF_\lambda\ni s'\ge s$ is strong and $a \in A(s)$ we have $s\le s'\leq s'^a = s^a$, so by \autoref{dfn:CloseKnit}\autoref{item:CloseKnitLongDecreasingSequence} applied to $s'$, the meet $n(s)$ is indeed defined.
Note that for $t\le s$ we have $n(t)\ge n(s)$.

Suppose there is no greatest $n(s)$ for strong $s$. Choose strong $s_0\in\fF_\lambda$ and let $\mu = \mu_{s_0}$ as per \autoref{dfn:CloseKnit}\autoref{item:CloseKnitLongDecreasingSequence}.
  Since $\mu < \kappa$, we may then construct by induction a sequence $(s_\alpha)_{\alpha \leq \mu}$ of strong elements, starting with $s_0$.
  At successor stages take some strong $t$ such that $n(t) \nleq n(s_\alpha)$ and let $s_{\alpha+1} = s_\alpha \wedge t$, so $n(s_{\alpha+1})\ge n(t)$ and $n(s_{\alpha+1}) \ge n(s_\alpha)$, whence $n(s_{\alpha+1}) > n(s_\alpha)$. 
  At limit stages put $s_\alpha = \bigwedge_{\beta<\alpha} \, s_\beta$.
  If $a \in A(s_\mu)$ then for every $\alpha \leq \mu$ we have $a \in A(s_\alpha)$ and
  \begin{gather*}
    s_0 \leq n(s_0) \leq n(s_\alpha) \leq s_\alpha^a = s_0^a.
  \end{gather*}
  This produces a chain of length $\mu$ between $s_0$ and $s_0^a$, contradicting the choice of $\mu$.
  Therefore there exists a greatest $n(s) \in \fL$; as it must be unique, it is a fixed point of $\Gamma$.
\end{proof}
If $\fI$ is finite then $\lambda = \aleph_0$; if all $\mu_s$ are finite, we can also take $\kappa = \aleph_0$ (in particular, condition \autoref{item:CloseKnitDeltaCompact} of \autoref{dfn:CloseKnit} holds automatically), and \autoref{thm:CloseKnit} allows us to recover \cite[Theorem 1]{wa98} at least qualitatively. The finite index/difference/codimension versions of Schlichting's Theorem follow.

\section{Almost invariant families in continuous logic}

Recall that in a $\kappa$-saturated metric structure $\M$ with domain $M$, a subset of $M^n$ is \emph{type-definable} if it is the intersection of the zero-sets of fewer than $\kappa$ many formul\ae.
Given a set $F \subseteq M^\alpha \times M^\beta$ and $a \in M^\beta$, we define $F_a = \bigl\{ b : (b,a) \in F \bigr\}$.
A family $\fF$ of subsets of $M^\alpha$ is \emph{type-definable} if there are type-definable sets $F \subseteq M^\alpha \times M^\beta$ and $A \subseteq M^\beta$ such that $\fF = F_A =  \{F_a : a \in A\}$.

\begin{dfn}
  We shall consider a type-definable ambient object $X$ of one of three kinds: sets, groups or vector-spaces over a definable field $K$.
  For two type-definable sub-objects of the same kind $S,S' \subseteq X$, we say that $S$ is \emph{commensurably contained} in $S'$ if
  \begin{itemize}
  \item $|S\setminus S'|$, or
  \item $[S: S \cap S']$, or
  \item $\codim_S (S \cap S')$, respectively,
  \end{itemize}
  is strictly less than $\kappa$.
  We say that $S$ and $S'$ are \emph{commensurable} if either is commensurably contained in the other.
\end{dfn}

Note that by $\kappa$-saturation this implies that the difference/index/codimension does not increase when we replace $\M$ by an elementary extension; we say that it is {\em bounded}.
In the following, whenever we talk about type-definable sets, we shall assume that $\M$ is sufficiently saturated.
\begin{dfn}
  Let $X$ be a type-definable set/group/vector space over $K$ in a metric structure $\M$.
  A type-definable family $\fF$ of subsets/subgroups/subspaces of $X$ is {\em almost invariant} if $F$ and $F'$ are commensurable for all $F,F'\in\fF$.
  For an almost invariant family $\fF$, we shall say that a subset/subgroup/subspace $S$ is commensurable with $\fF$ if $S$ is commensurable with some (equivalently, all) $F\in\fF$.
\end{dfn}
\begin{rmk} If $\fF$ is an almost invariant type-definable family, then by compactness the commensurability must be {\em uniform}: There is a cardinal $\lambda<\kappa$ such that the difference/index/codimension is bounded by $\lambda$ for all $F,F'\in\fF$. This implies that if $S$ is commensurable with $\fF$, it is {\em uniformly} commensurable with $\fF$, i.e.\ the difference/index/codimension is bounded independently of $F\in\fF$.\end{rmk}

\begin{thm}\label{thm:schlichting}
  Let $X$ be a type-definable set/group/vector space over a definable field 
  $K$ in a metric structure $\M$, and $\Gamma$ a type-definable group of automorphisms of $X$.
  Suppose $\fF$ is a $\Gamma$-invariant almost invariant type-definable family of subsets/subgroups/subspaces of $X$.
  Then there is a $\Gamma$-invariant subset/subgroup/subspace $N$ of $X$ commensurable with all $F\in\fF$, which is moreover type-definable over the same parameters.
\end{thm}
\begin{proof}
  We may assume that $X$, $\Gamma$ and $\fF = \{F_a : a \in A\}$ are type-defined over $\emptyset$, and $K$ is definable over $\emptyset$.
  Notice that then we may also enumerate the family $\fF$ as $\{F_{a,\gamma} : a \in A, \gamma \in \Gamma\}$ where $F_{a,\gamma} = \gamma F_a$.
  Let $S\subseteq X$ be type-definable.
  Suppose $F_a$ is defined (for $a \in A$) by $\Phi(x,a)=0$, where $\Phi$ is a family of $[0,1]$-valued formul\ae\ closed under the connective $\max$, and $|\Phi| < \kappa$.
  For $\gamma \in \Gamma$, $n \in \bbN$ and $\varphi \in \Phi$ define
  \[\begin{array}{ll}
  \bullet\mbox{ in the set case:}&
      \delta_{\varphi,n}(S,a,\gamma) = \sup_{x \in S^n}\,\bigl(\min_{i < j < n} d(\gamma^{-1} x_i,\gamma^{-1} x_j) \wedge \min_{i<n} \varphi(\gamma^{-1} x_i,a)\bigr),\\
  \bullet\mbox{ in the group case:}&
      \delta_{\varphi,n}(S,a,\gamma) = \sup_{x \in S^n} \, \min_{i < j < n} \varphi\bigl( \gamma^{-1}(x_i^{-1} x_j),a \bigr),\\
  \bullet\mbox{ in the vector space case:}&
     \delta_{\varphi,n}(S,a,\gamma) = \sup_{x \in S^n} \, \inf_{\eta \in K^n\setminus\bar0}\varphi\bigl(\gamma^{-1}(\sum \eta_i x_i), a \bigr).
    \end{array}\]
  Let $\fI = [0,1]^{\Phi \times \bbN}$, so $\lambda = |\Phi|^+ + \aleph_0$.
  For type-definable $S \subseteq X$ and $(a,\gamma) \in A \times \Gamma$ define
  \begin{gather*}
    \delta(S,a,\gamma) = \bigl( \delta_{\varphi,n}(S,a,\gamma) : (\varphi,n) \in \Phi \times \bbN \bigr)\in\fI.
  \end{gather*}
  Let $\fL$ be the lower semi-lattice of type-definable subsets/subgroups/subspaces commensurable with $\fF$.
  Then conditions \autoref{item:CloseKnitDeltaContinuous} and \autoref{item:CloseKnitDeltaCompact} of \autoref{dfn:CloseKnit} hold by compactness.

For $S\in\fL$ and $F \in \fF$ put $S^F=S\cup F$ in the set case, and $S^F=S+F$ in the vector space case. In the group case, put
$$S^F=\bigcap_{s\in S}(SF)^s$$
and note that $S^F$ is a supergroup of $S$ with $[S^F:S]$ bounded. Moreover, since $S$ and $F$ are commensurable, there is a bounded set $I\subseteq S$ with $S=(S\cap F)I$. Then 
$S^F=\bigcap_{s\in S}(SF)^s=\bigcap_{s\in I}(SF)^s$, so $S^F$ is type-definable,
Finally, put $S^{a,\gamma}=S^{F_{a,\gamma}}$.

We claim that in all three cases, condition \autoref{item:CloseKnitIncrement} of \autoref{dfn:CloseKnit} holds. Clearly $S\subseteq S^F$. So assume that $T \subseteq S$ and $T^F \neq S^F$, where $F = F_{a,\gamma}$.

  In the set case, since $T\cup F\subseteq S\cup F$, this means that there exists $x \in S \setminus (T \cup F)$.
  In particular, $d(\gamma^{-1} x,\gamma^{-1} T) > 0$, so for some $\varphi \in \Phi$ and $0<e<1$ we have $d(\gamma^{-1} x,\gamma^{-1} y) \wedge \varphi(\gamma^{-1} x, a) \geq e$ for all $y \in T$. Since $T$ and $F$ are commensurable, there exists $n$ such that $\delta_{\varphi,n}(T,a,\gamma) < e$, and we may assume that $n$ is least such. As
$\delta_{\varphi,0} \equiv 1$, we have $n>0$, and $\delta_{\varphi,n-1}(T,a,\gamma) \geq e$ by minimality of $n$. Therefore $\delta_{\varphi,n}(S,\gamma) \geq e$, so $\delta(T,F) < \delta(S,F)$.

 In the group case, note that if $T\le S$ with $SF=TF$ and $I\subset T$ is a system of representatives for $S/(S\cap F)\cong SF/F=TF/F\cong T/(T\cap F)$, then
$$S^F=\bigcap_{s\in S}(SF)^s= \bigcap_{s\in I}(SF)^s=\bigcap_{s\in I}(TF)^s=  \bigcap_{s\in T}(TF)^s= T^F.$$
It hence suffices to show that $S\subseteq TF$, as then $SF=TF$. Suppose not, and consider $x \in S \setminus TF$.
  In other words $y^{-1}x \notin F$ for all $y \in T$.
  By compactness, for some $\varphi \in \Phi$ and $0<e<1$ the partial type $y \in T$ implies that $\varphi\left( \gamma^{-1}(y^{-1} x) ,a \right) \geq e$.
  We conclude as above.

  In the vector space case, since $T+F\le S+F$, this means that there is $x\in S\setminus(T+F)$.
That is, $\eta x+y\notin F$ for all $\eta\in K^\times$ and $y\in T$. By compactness, for some $\varphi \in \Phi$ and $0<e<1$ the partial type $y \in T$ implies that $\inf_{\eta\in K^\times}\varphi\bigl(\gamma^{-1}(\eta x+y) ,a \bigr) \geq e$.
  Again we conclude as above.

  If $F\in\fF$ and $S\in\fF_\lambda$, then $S^F$ is type-definable with strictly less than $\lambda$ parameters.
  Since $F$ and $S$ are commensurable, the difference $|S_F\setminus S|$, the index $[S_F:S]$ or the co-dimension $\codim_{S^F}(S)$ are bounded by $2^{<\lambda}$.
  It follows that any chain between $S$ and $S_F$ has length at most $2^{<\lambda}$.
  We can thus put $\kappa=(2^{<\lambda})^+$ to satisfy condition \autoref{item:CloseKnitLongDecreasingSequence} of \autoref{dfn:CloseKnit}.

Define an action of $\Gamma$ on $(A\times\Gamma)$ by $\gamma'\cdot(a,\gamma)=(a,\gamma'\gamma)$. Then 
$$\begin{array}{lll}\gamma'F_{a,\gamma}&=\gamma'\gamma F_a&=F_{a,\gamma'\gamma}=F_{\gamma'\cdot(a,\gamma)},\\
\delta_{\varphi,n}(\gamma'S,\gamma'\cdot(a,\gamma))&=\delta_{\varphi,n}(\gamma'S,a,\gamma'\gamma)&=\delta_{\varphi,n}(S,a,\gamma),\\
(\gamma' S)^{F_{\gamma'\cdot(a,\gamma)}}&=(\gamma' S)^{\gamma'F_{a,\gamma}}&=\gamma'(S^{F_{a,\gamma}}),\end{array}$$
so everything is $\Gamma$-invariant. Clearly, we also have invariance under $\Aut(\fM)$.

By \autoref{thm:CloseKnit} there is some $N\in\fL$ invariant under the group of automorphisms of $\fL$ generated by $\Gamma \cup \Aut(\fM)$.
In particular $N$ is commensurable with $\fF$, type-definable over $\emptyset$, and $\Gamma$-invariant.
\end{proof}

\begin{remark} The usual metrisation of quotients modulo type-definable equivalence relations shows that \autoref{thm:schlichting} also holds for hyperdefinable families of commensurable subsets/subgroups/sub-vector spaces. (If the equivalence relation is given by an uncountable partial type, we first express the quotient as a type-definable subset of an infinite (possibly uncountable) product of hyperimaginary sets modulo countable equivalence relations, which in turn are equivalent to imaginary metric sorts.)
\end{remark}

\section{Fields}
For two fields $F$ and $K$ we say that $F$ is commensurably contained in $K$ if the degree $[FK:K]$ is finite. Then commensurable fields form an {\em upper} semi-lattice which need not be closed under meet. 
Theorem 1 applied to Example 1.(v) in \cite{wa98} implies in particular that if $K$ is a field, $\Gamma$ a group of automorphisms of $K$ and $\fF$ a family of uniformly commensurable subfields of $K$ such that any finite intersection of elements in $\fF$ is commensurable with $\fF$, then there is a $\Gamma$-invariant subfield of $K$ commensurable with $\fF$. However, the condition that finite intersections be commensurable with $\fF$ is much stronger than mere pairwise commensurability. In this section, we shall show that in case the extensions $FF'/F$ for $F,F'\in\fF$ are separable, or the Er\v sov invariant of any field in $\fF$ is finite, there still is a $\Gamma$-invariant commensurable subfield.

\satz\label{thm:fields} Let $K$ be a field, $\Gamma$ a group of automorphisms of $K$ and $\fF$ a $\Gamma$-invariant family of uniformly commensurable subfields of $K$. If $FF'/F$ is separable for all $F,F'\in\fF$, or if the Er\v sov invariant $[F:F^p]$ is finite for any $F\in\fF$, there is a $\Gamma$-invariant subfield $N$ commensurable with $\fF$.\esatz
\bew Clearly we may assume that $\fF$ consists of a single orbit under $\Gamma$. For a subfield $F\le K$ let $F^s$ be the separable closure of $F$ in $K$,
 and put $ K_s=\bigcap_{F\in\fF}F^s$, a $\Gamma$-invariant subfield of $K$ satisfying $K_s=K_s^s$. If $FF'/F$ is separable for all $F,F'\in\fF$, then $F\le K_s=F^s$ for all $F\in\fF$, and $K_s$ is Galois over $F\cap K_s=F$. We put $\fF_s=\fF$. 
 
Otherwise, by uniform commensurability of $\fF$ there is a finite power $q$ of $p$ such that $F^qF'/F'$ is separable for all $F,F'\in\fF$, where $p=\mbox{char}(K)>0$ is the characteristic. Then $F^q\le F\cap K_s\le F\le F^s$ for all $F\in\fF$; note that $F^s$ is normal over $F^q$, whence over $F\cap K_s$, and $F$ is the pure inseparable closure of $F\cap K_s$ inside $F^s$. Hence $F^s=F(F\cap K_s)^s$. Since $(F\cap K_s)^s\le K_s^s=K_s\le F^s$, we have $(F\cap K_s)^s=K_s$, so $K_s$ is Galois over $(F\cap K_s)$.
Moreover, as $F$ has finite Er\v sov invariant, $[F:F^q]$ is finite; it is independent of $F$ since $\fF$ consists of a single $\Gamma$-orbit. Thus $\fF_s=\{F\cap K_s:F\in\fF\}$ is a family of uniformly commensurable subfields of $ K_s$.

Let $G=\aut(K_s)$ with the topology of pointwise convergence and the induced action of $\Gamma$. For $F\in\fF$ put $H_F=\Gal(K_s/F\cap K_s)\le G$. Then $F\cap K_s=\Fix_{K_s}(H_F)$.

Now $H_{\gamma F}=H_F^{\gamma^{-1}}$ for $\gamma\in\Gamma$. Uniform commensurability of $\fF_s$ implies that $\fH=\{H_F:F\in\fF\}$ is a $\Gamma$-invariant family of uniformly commensurable closed subgroups of $G$. By Schlichting's Theorem \cite[Theorem 6(iii)]{bl} applied to the family $\fH$ of subgroups of $G\rtimes\Gamma$, there is a $\Gamma$-invariant subgroup $H\le G$ commensurable with $\fH$. Moreover $H$ is closed as it is a finite extension of a finite intersection of groups in $\fH$. Then $N=\Fix_{K_s}(H)$ is a $\Gamma$-invariant subfield of $ K_s$ commensurable with $\Fix_{K_s}(H_F)=F\cap K_s$ for any $H_F\in\fH$. As $F\cap K_s$ is commensurable with $F$, we are done.\ebew

\cite{wat04} gives examples of commensurable fields $F$ and $F'$ such that $FF'/F$ is purely inseparable, $FF'/F'$ is either separable or purely inseparable, and $F\cap F'$ has infinite degree in $F$ and in $F'$. As we have not been able to deal with this problem, we have not managed to prove \autoref{thm:fields} in full generality.


\begin{thebibliography}{9}
\bibitem{bl}George M. Bergman and Hendrik W. Lenstra, Jr.,
  \newblock Subgroups close to normal subgroups,
  \newblock{\em J. Alg.} 127:80--97, 1989.
\bibitem{bpp}Leonid Brailovsky, Dimitrii V. Pasechnik and Cheryl E. Praeger,
  \newblock Subsets close to invariant subsets for group actions,
  \newblock{\em Proc.\ Amer.\ Math.\ Soc.} 123:2283--2295, 1995.
\bibitem{Udi}Ehud Hrushovski,
\newblock Stable group theory and approximate subgroups,
\newblock{\em J. Amer.\ Math.\ Soc.} 25(1):189--243, 2012.
\bibitem{neu} Peter Neumann,
  \newblock Permutation groups and close-knit families of sets,
  \newblock{\em Arch.\ Math.\ (Basel)} 67:265--274, 1996.
\bibitem{sch} G\"unter Schlichting,
  \newblock Operationen mit periodischen Stabilisatoren,
  \newblock{\em Arch.\ Math.\ (Basel)} 34:97--99, 1980.
\bibitem{wa98} Frank O. Wagner,
  \newblock Almost invariant families,
  \newblock{\em Bull.\ London Math.\ Soc.} 30:235-240, 1998.
\bibitem{wa00} Frank O. Wagner,
  \newblock{\em Simple Theories},
  \newblock Mathematics and its Applications 503.
  Kluwer Academic Publishers, Dordrecht, 2000.
\bibitem{wa16} Frank O. Wagner,
  \newblock The right angle to look at orthogonal sets,
  \newblock{\em J. Symb.\ Logic}, 81(4):1298-1314, 2016.
  \bibitem{wat04}William C. Waterhouse.
  \newblock Intersections of two cofinite subfields,
  \newblock{\em Archiv Math. (Basel)} 82(4):298--300, 2004.
\end{thebibliography}
\end{document}